\documentclass{amsart}
\usepackage[utf8]{inputenc}
\usepackage{amssymb,latexsym}
\usepackage{amsmath}
\usepackage{graphicx}
\usepackage{textcomp}
\usepackage[nobysame, alphabetic]{amsrefs}

\usepackage{amsthm,amssymb,enumerate,graphicx, tikz}
\usepackage{amscd}
\usepackage{setspace}
\usepackage{comment}
\usepackage{hyperref}
\usepackage[capitalise]{cleveref}
\usepackage{subcaption}

\newtheorem{theorem}{Theorem}[section]

\newtheorem*{theorem*}{Theorem}
\newtheorem{proposition}[theorem]{Proposition}

\newtheorem{corollary}[theorem]{Corollary}

\theoremstyle{definition}

\theoremstyle{remark}

\newtheorem{example}[theorem]{Example}

\newcommand{\R}{\mathbb{R}}
\newcommand{\rr}{\mathbb{R}}
\newcommand{\zz}{\mathbb{Z}}

\newcommand{\conv}{\textrm{conv}}

\title{Improved Tverberg theorems for certain families of polytopes}

\author{Pablo Sober\'on}\thanks{P. Sober\'on: Department of Mathematics, Baruch College and The Graduate Center, City University of New York, USA.  \url{psoberon@gc.cuny.edu}.  The research of P. Sober\'on was supported by NSF CAREER award no. 2237324, NSF award no. 2054419 and a PSC-CUNY Trad B award.} 
\author{Shira Zerbib}\thanks{S. Zerbib: Department of Mathematics, Iowa State University, USA.  \url{zerbib@iastate.edu}. The research of S. Zerbib was supported by NSF CAREER award no. 2336239,
NSF award no. DMS-1953929,
and Simons Foundation award no. MP-TSM-00002629.}

\begin{document}

\begin{abstract}
A theorem of Gr\"unbaum, which states that every $m$-polytope is a refinement of an $m$-simplex, implies the following generalization of Tverberg's theorem: if $f$ is a linear function from an $m$-dimensional polytope $P$ to $\R^d$
and $m \ge (d + 1)(r - 1)$, then there are $r$ pairwise disjoint faces of $P$
whose images intersect. Moreover, the topological Tverberg theorem implies  that this statement is true whenever the map $f$ is continuous and $r$ is a prime power. In this note we show that for certain families of polytopes the lower bound on the dimension $m$ of the polytopes can be significantly improved, both in the affine and topological cases. 
\end{abstract}



\maketitle

\section{Introduction}

Tverberg's theorem is a fundamental result in discrete geometry.  It gives combinatorial information about the overlaps of linear maps from high-dimensional simplices to low-dimensional real spaces.  There are now a myriad extensions and modifications of this result (see, e.g. \cites{Blagojevic:2017bl, DeLoera:2019jb, Barany:2018fy, Barany2022}).

\begin{theorem}[Tverberg 1966 \cite{Tverberg:1966tb}]
    Let $r,d$ be positive integers, and  $m=(r-1)(d+1)$. Let $\Delta^m$ be the $m$-dimensional simplex.  Then, for any linear map $f:\Delta^m \to \rr^d$, there exist $r$ points $x_1, \dots, x_r$ in pairwise disjoint faces such that $f(x_1) = \dots = f(x_r).$
\end{theorem}

The topological versions of Tverberg's theorem, in which the map $f$ is only required to be continuous instead of linear are of particular interest.  Establishing Tverberg's theorem for continuous maps has motivated significant developments in topological combinatorics.  The topological version of Tverberg's theorem holds when $r$ is a prime power \cite{Volovikov:1996up}, and the condition is necessary \cites{Frick:2015wp, Mabillard.2015, Avvakumov2023}. 

By a theorem of Gr\"unbaum \cite{grunbaum:book}, every $m$-dimensional polytope is a refinement of $\Delta^m$. Therefore, the following generalization of of Tverberg's theorem holds:  

\begin{theorem}\label{gentverberg}
      Let $d,r$ be positive integers and $P$ be a polytope of dimension at least $(d+1)(r-1)$.  For any linear function $f: P \to \R^d$ there exist points $x_1, \dots, x_r$ in pairwise vertex-disjoint faces of $P$ such that $f(x_1) = \dots = f(x_r)$. Moreover, the statement is true when $r$ is a prime power and $f$ is assumed to be continuous (but not necessarily linear). 
\end{theorem}

This was first asked by Tverberg \cite{gruberetal} (and more recently by B\'ar\'any and Kalai \cite{Barany2022}), and observed to be true by Hasui, Kishimoto, Takeda and Tsutaya \cite{Hasuietal}.
  
 The purpose of this paper is to show that  several families of polytopes, and certain values of $r$, the lower bound on the dimension $m$ of the polytope in Theorem \ref{gentverberg} can be significantly improved.  Many of our results hold for continuous functions.
\medskip

Our main results, and the organization of the paper are as follows:
\begin{itemize}
    
    \item In \cref{sec:cross} we prove  that for cross polytopes of dimension  $m\ge (d+1)(r-1)/2$ and $r$ prime and when $f$ is continuous,  the number of  sets of $r$  pairwise disjoint faces whose images under $f$ intersect is at least $ 
    \frac{1}{r!}\left(\frac{r-1}{2}\right)^{(d+1)(r-1)/2}$, even if we are forbidden from using a particular fixed vertex.
    \item \cref{sec:subdivisions} is devoted to  polytopes with small face diameter. 
    We prove that when $r$ is prime, the dimension $m$ in the topological version of \cref{gentverberg} can be improved to $(r-1)d+2$ if the face diameter of the polytope is at most $\pi/r$.  Further, for such polytopes we get an improvement on $m$ for general values of $r$, provided that $r$ is large enough with respect to the dimension $d$.
   To this end we prove a Borsuk--Ulam type theorem (\cref{thm:BU}) that may be of interest on its own. 

    \item In \cref{sec:neighborly} we show that  for $r\ge 3$, the statement in Theorem \ref{gentverberg}  holds for any cyclic polytopes of dimension at least $2(d+1)$ that has  at least $(r-1)(d+1)+1$ vertices. 
    \item Finally, in \cref{sec:dimone} we prove  that in the case $d=1$, if the 1-skeleton of the polytope $P$ is triangle-free, then the bound on $m$ in the topological version of \cref{gentverberg} can be improved to $m\le r$. 
\end{itemize}

 \section{Cross polytopes}\label{sec:cross}

Another proof of Theorem \ref{gentverberg}  for cross polytopes follows from  a theorem of Blagojevi\'c, Matschke, and Ziegler, known as the optimal colorful Tverberg theorem.

\begin{theorem}[Blagojevi\'c, Matschke, Ziegler \cites{Blagojevic:2011vh, Blagojevic:2015wya}]\label{thm:BMZ}
Let $p$ be a prime number, and $n,d$ positive integers. Given $(p-1)(d+1)+1$ points in $\R^d$ colored with $n$ colors such that each color class has at most $p-1$ points, there exists a partition of the points into sets $A_1, \dots, A_p$, such that $\bigcap_{j=1}^p\conv A_j \neq \emptyset$, and $A_j$ contains at most one point of every color class.     
\end{theorem}

To see that this theorem implies Theorem \ref{gentverberg} for a cross polytope  $P$ (with a linear map $f$) observe that, By Bertrand's postulate, there exists a prime number $p$ such that $r \le p \le 2r-3$, and thus the number of vertices of $P$ is $2(r-1)(d+1) \ge (p-1)(d+1)+1$.  Choose  $(p-1)(d+1)+1$ vertices of $P$ and color them with up to $(r-1)(d+1)$ colors, so that each pair of antipodal vertices of $P$ receive the same distinct color. Thus   each color class has at most $2 \le p-1$ vertices.  Now apply \cref{thm:BMZ} to obtain $p$ points $x_1, \dots, x_p$ such in pairwise disjoint faces, such that each of the minimal faces containing $x_1, \dots, x_p$ uses at most one point of each color and $f(x_1) = \dots = f(x_p)$ and thus is a face of the $P$.  Since $p \ge r$, we are done.

In this section we show that for cross polytopes, when $r$ is prime, the dimension $m$ in the statement of Theorem \ref{gentverberg} can be improved  to  $1+(r-1)(d+1)/2$, even when the function $f$ is assumed to be continuous. Moreover, in this case we give a 
 lower bound on the number of $r$-sets of pairwise disjoint faces whose images intersect, even when we are allowed to fix a vertex that will not be contained in any of those faces. 

 %The same use of Bertrand's postulate as above, shows that \cref{thm:cross-polytope-bound} implies \cref{thm:cross-polytopes-solved}.

% \textcolor{red}{In the next theorem, can we assume $m \ge 1+(r-1)(d+1)/2$, or do we have to have equality? } \textcolor{orange}{inequality suffices, since we can ignore vertices if there are more.  We can also state the theorem for a suspension of a cross-polytope and a point (if we don't want the ``ignore a vertex'' bit} \textcolor{red}{ Thanks, I changed. I actually like the current statement. It shows that maybe some more careful counting  can give a better bound.}

\begin{theorem}\label{thm:cross-polytope-bound}
       Let $r$ be a prime number, and $P$ be a cross-polytope of dimension $m \ge 1+(r-1)(d+1)/2$. Let  $f:P \to \rr^d$ be a continuous function, and $p$ be a vertex of $P$.  Then, the number of sets $\{A_1, \dots, A_r\}$ of $r$ pairwise disjoint faces of $P$ such that do not  contain $p$ and
    \[
    \bigcap_{j=1}^r f(A_j) \neq \emptyset
    \]
    is at least 
    \[
    \frac{1}{r!}\left(\frac{r-1}{2}\right)^{m-1}.
    \]
\end{theorem}

The condition excluding a particular vertex $p$ may seem strange at first sight, but it is needed in order to make sure that the number of available vertices is $2m-1 = (r-1)(d+1)+1$, which is the number of points that Tverberg's theorem requires.  In other words, $m=(r-1)(d+1)/2+1$ is the smallest dimension for which we might expect the theorem above to hold, and forcing us not to use a vertex is the strongest condition of this type we can impose while these partitions continue to exist.  %The number of vertices is small enough to imply the B\'ar\'any--Kalai conjecture for all cross-polytopes, and even giving a large lower bound on the number desired of $r$-fold overlaps.

To put the bound into perspective, we can compare it with known results for Tverberg's theorem.  For $r$ prime, the current best bound on the number of Tverberg partitions of a set of $(r-1)(d+1)+1$ points is $\frac{1}{r!}(r/2)^{m-1}$, using the notation of \cref{thm:cross-polytope-bound}.  Therefore, by imposing all these conditions we are reducing the bound on the number of partitions by a factor of 
\[
\left(\frac{r}{r-1}\right)^{(r-1)(d+1)/2} \sim e^{(d+1)/2},
\]
which surprisingly depends only on the dimension.

The main topological result we will need is Dold's generalization of the Borsuk--Ulam theorem

\begin{theorem}[Dold 1983 \cite{Dold:1983wr}]\label{thm:dold}
Let $G$ be a finite group, $|G|>1$, $X$ be an $n$-connected space with a free action of $G$, and $Y$ be a paracompact topological pace of dimension at most $n$ with a free action of $G$.  Then, there exists no $G$-equivariant continuous map $f:X \to_G Y$. 
\end{theorem}

\begin{proof}[Proof of \cref{thm:cross-polytope-bound}]
    We use a modification of the ``deleted join'' method that was used to prove Tverberg's theorem by Vu\v{c}i\'c and \v{Z}iveljevi\'c \cite{Vucic1993be}.  We will count  ordered $r$-tuples of  faces $(A_1, \dots, A_r)$.

    Let $p_1, \dots, p_{2m}$ be the vertices of $P$, so that $p_{2k-1}, p_{2k}$ are antipodal for every $k$.  Assume without loss of generality that the forbidden vertex is $p_{2m}$.

    To use the test map / configuration space scheme, we need to construct a topological space $K$ that parametrizes the candidates for overlapping $r$-tuples.  For each $p_i$, let $G_i$ be a copy of the discrete set $\{1,2,\dots, r\}$, which we denote as $G_i :=\{1^{(i)}, \dots, r^{(i)}\}$ for convenience.  We are going to use $G_i$ to determine which set $p_i$ belongs to.

    for $i=1,\dots, m-1$, we choose two distinct numbers $\{a_i, b_i\}$ in $[r-1]$.  We have $\binom{r-1}{2}$ way to do this for each $i$, which gives us a total of $\binom{r-1}{2}^{m-1}$ possibilities.

    Using $a_i, b_i$, we construct a subset $M_i$ of the topological join $G_{2i-1}*G_{2i}$ as 
    \[
    M_i = \{x^{(2i-1)}* y^{(2i)}: y-x \mbox{ is congruent to one of }a_i,b_i \mbox{ modulo }r\}.
    \]
    Note that since $a_i, b_i$ are both different from $0$, this implies that $x$ and $y$ will not be congruent modulo $r$, which will force $p_{2i-1}, p_i$ to be on different sets at the end of the construction.

    Each join $x^{(2i-1)}* y^{(2i)}$ is a segment in $M_i$ connecting a vertex of $G_{2i-1}$ to a vertex of $G_{2i}$.  Topologically, $M_i$ is a graph where all vertices have degree equal to two, so it must be a union of disjoint cycles.  Since $r$ is prime and the construction is cyclic on $[r]$, we have that $M_i$ is a single cycle and therefore connected.

    Recall that if two simplicial complexes $U$, $V$ are $i$-connected and $j$-connected respectively, then their join $U * V$ is at least $(i+j+2)$-connected.  Let $$K = M_1 * M_2 * \dots * M_{m-1} * G_{2m-1}.$$  Since each $M_i$ is $0$-connected and $G_{2m-1}$ is $(-1)$-connected (non-empty), then $K$ is at least $(2m-3)$-connected.

    Now we construct a test map $h$, that will determine whether the $r$-tuple of sets $A_1, \dots, A_r$ induced by an element of $K$ satisfies the conditions of the theorem.  After doing that, we will count how many different $r$-tuples we can generate by changing the choices of $\{a_i, b_i\}$.
    An element of $K$ is a formal convex combination of the form
    $\sum_i \alpha_i (x_i)^{(i)}$ for some coefficients $\alpha_i \ge 0$ that sum to $1$ and values $x_i \in [r]$.  By the construction of $K$, we have $x_{2i}-x_{2i-1}$ is congruent to either $a_i$ or $b_i$ modulo $r$ for each $i \in [m-1]$.

    Let $I_1, \dots, I_r$ be sets of indices where $I_{j} = \{i: x_i = j\}$.  By the construction of $K$, the set $I_j$ cannot contain two indices corresponding to antipodal vertices of $P$.  Now consider the formal sum $\sum_{i \in I_j} \alpha_i (x_i)^{(i)} = \sum_{i \in I_j} \alpha_i j^{(i)}$, which we use to construct $h$.  Consider
    $\beta_j = \sum_{i \in I_j} \alpha_i$.  For $j \in [r]$, we define a vector $q_j \in \rr^{d+1}$ as follows.
    \[
    q_j = \begin{cases}
        0 \in \rr^{d+1} & \mbox{ if } \beta_j = 0 \\
        \left(\beta_j f\left(\frac{1}{\beta_j}\left(\sum_{i \in I_j}\alpha_i p_i\right)\right), \beta_j \right) & \mbox{ if } \beta_j \neq 0.
    \end{cases}
    \]

    In the second case, the input of $f$ is a convex combination of vertices of $P$ that contains no antipodal pairs, so it is a point of the boundary of $P$, and we can therefore evaluate $f$ on it.  The point $q_j$ varies continuously as we change the point of $K$ for the construction continuously.

    Finally, we define
    \begin{align*}
        h:K & \to \rr^{r(d+1)} \\
        \sum_{i} \alpha_i (x_i)^{(i)} & \mapsto (q_1, \dots, q_r)
    \end{align*}

    Both spaces $K$ and $\rr^{r(d+1)}$ have an action of $\zz_r$, corresponding on $K$ to shift the values $x_i$ cyclically, and in $\rr^{r(d+1)}$ shift the $(d+1)$-dimensional vectors $q_i$ cyclically.  With these actions, $h$ is equivariant.

    We claim that there exists a point in $K$ whose image satisfies $q_1 = \dots = q_r.$  Indeed, if we cannot find such a point, we apply standard dimension reduction arguments to reach a contradiction.
First, consider $W^{(r-1)(d+1)} = \{(x_1, \dots, x_r): x_i \in \rr^{d+1} \mbox{ for all } i, \sum_i x_i = 0\}$.  This is an $(r-1)(d+1)$-dimensional linear subspace of $\rr^{r(d+1)}$.  The orthogonal projection $\pi : \rr^{r(d+1)}\to W^{(r-1)(d+1)}$ simply maps each $r$-tuple $(q_1,\dots, q_r)$ to $(q_1-a,\dots, q_r-a)$ where $a = \frac{1}{r}\sum_{i}q_i$ is the average of the vectors.
Therefore, we are looking for a zero of the map $\pi \circ h$.  If this map has no zeros, we can then project the image from the origin onto its unit sphere $S(W^{(r-1)(d+1)})$ of dimension $(r-1)(d+1)-1=2m-3$.
Overall, we have constructed a map
    \begin{align*}
        g:K & \to S(W^{(r-1)(d+1)}) \\
        \bar{x} & \mapsto \frac{\pi \circ h (\bar{x})}{\|\pi \circ h (\bar{x})\|}
    \end{align*}
This map is $\zz_r$-equivariant.  Since $r$ is prime, the action of $\zz_r$ is free on both $K$ and $S(W^{(r-1)(d+1)})$.  Therefore, we obtain a contradiction to Dold's theorem.  
    
    If a point in $K$ satisfies $q_1 = \dots = q_r$ in the image of $h$, then we have $\beta_1 = \dots = \beta_r$, and the points $f\left(\frac{1}{\beta_j}\left(\sum_{i \in I_j}\alpha_i p_i\right)\right)$ are the same for all $j$. Thus the sets $A_1, \dots, A_r$ where $A_j = \{p_i : i \in I_j\}$ are the faces of the cross-polytope whose images overlap.

    Therefore, each choice of the $m-1$ sets $\{a_i, b_i\}$ gives us a good $r$-tuple of faces.  Let us count how many times a $r$-tuple of faces can be counted this way.  Given an $r$-tuple $(A_1,\dots,A_r)$,  if a point $p_{2k+1}$ is in $A_j$ and $p_{2k}$ is in $A_{j'}$, then one of $\alpha_k, \beta_k$ is equal to $j'-j$, so there are $r-2$ possible choices for the other element of the pair $\alpha_k, \beta_k$ .  Thus a fixed $r$-tuple is counted at most $(r-2)^m$ times.  The number of ordered partitions is therefore bounded above by
    \[
    \left(\frac{1}{r-2}\right)^m \binom{r-1}{2}^m = \left(\frac{r-1}{2}\right)^m.
    \]
    Finally, we divide by $r!$ to get a bound on the number of unordered partitions.
\end{proof}

\section{Polytopes with small face diameter}\label{sec:subdivisions}

In this section, we prove that when $r$ is prime, the dimension $m$ in the topological version of \cref{gentverberg} can be improved to $(r-1)d+2$ if the diameter of the faces of $P$ is small enough. Further, for such polytopes we get an improvement on $m$ also for general values of $r$, provided that $r$ is large enough with respect to $d$.

Let $P$ be a $m$-polytope and fix a homeomorphism $h:\partial(P)\to S^{m-1}$. We say that the diameter of a face $F$ of $P$ is $d$ if the spherical diameter of $h(H)$ in $S^{m-1}$ is $d$.
Given a polytope $P$ of dimension $m$, we denote by $P^{(k)}$ the polytope obtained by taking $k$ barycentric subdivisions of $P$. More precisely, given a homeomorphism $h:\partial(P)\to S^{m-1}$, let $T$ be the complex obtained by taking  $k$ barycentric subdivision of $h(\partial(P))$ and let $P^{(k)}$ be the polytope we get by applying  suitable perturbations of the vertices of $h^{-1}(Q)$ to make it a convex polytope. 

%We first show that for polytopes $P$ of dimension $m=(r-1)d+2$ with $r$ being prime,  $P^{(k)}$ satisfies an improved topological Tverberg theorem whenever $k\ge m\log (2r)$. 

We shall prove: 
\begin{theorem}\label{thm:small-facets}
    Let $d$ be a positive integer, $r$ be a prime, and $m \ge (r-1)d+2$.  Let $P$ be an $m$-dimensional polytope with face diameter smaller than $\pi/r$.  Then for any continuous map $f:P \to \rr^d$ there are $r$ pairwise vertex-disjoint faces whose images under $f$ intersect.
\end{theorem}

For general polytopes $P$, we have the following corollary.

\begin{corollary}\label{cor:prime}
    Let $d,r,m,k$ be positive integers where $r$ is a prime,  $m \ge (r-1)d+2$ and $k \ge m\log (2r)$.  Let $P$ be an $m$-dimensional polytope.  Then for any continuous map $f:P^{(k)} \to \rr^d$ there are $r$ pairwise vertex-disjoint faces whose images under $f$ intersect.
\end{corollary}
\begin{proof}
    The maximum diameter of a face in $\partial P$  is $2\pi$, and every subdivision reduces the diameter by a factor of $(m-1)/m$. Since $(m-1)/m \le e^{-1/m}$, we have that the face diameter of $P^{(k)}$ is at most $2\pi [(m-1)/m]^k \le \pi / r$. Now apply    \cref{thm:small-facets}.
\end{proof}

%\medskip
Further, whenever $r$ is  large enough with respect to  $d$, we can get an improvement on $m$ even if $r$ is not prime. The following corollary  improves the dimension $m$ in the topological version of \cref{gentverberg} as long as $r^{7/11} < (r-2)/d$, which happens  when $r$ is asymptotically larger than $\Omega(d^{11/4})$.

\begin{corollary}
    Let $d,r,m$ be positive integers such that ${m\ge (r + r^{7/11}-1)d+1}$. 
 Let $P$ be an $m$-dimensional polytope  with face diameter smaller than $\pi/r$.  Then for any continuous map $f:P \to \rr^d$ there are $r$ pairwise  vertex-disjoint faces whose images under $f$ intersect.
\end{corollary}
 
\begin{proof}
    By \cite{Lou-Yao:1992} there is a prime $p$ satisfying \[r\le p < r+r^{7/11}\] and thus 
    \[
    (p-1)d+2 \le (r+r^{7/11}-1)d+1.
    \]
    Let $Q$ be the $[(p-1)d+2]$-skeleton of $P$. By \cref{thm:small-facets}  for any continuous map $f:Q \to \rr^d$ there are $p \ge r$ points of pairwise vertex-disjoint faces such that $f(x_1) = \dots = f(x_p)$. Taking any $r$ of these points implies the conclusion, as any continuous map $f:P \to \rr^d$ induces a continuous map $f:Q \to \rr^d$, and the $[(p-1)d+2]$-skeleton of $P$ is $Q$.
\end{proof}

Again, if we apply a similar  argument to the one in Corollary \ref{cor:prime} we get the following result for the polytope obtained by taking enough subdivisions of a general polytope.

\begin{corollary}
       Let $d,r,m,k$ be positive integers so that ${m\ge (r + r^{7/11}-1)d+1}$ and  $k\ge m\log(2r)$. 
 Let $P$ be an $m$-dimensional polytope.  Then for any continuous map $f:P^{(k)} \to \rr^d$ there are $r$  pairwise vertex-disjoint faces whose images under $f$ intersect.
\end{corollary}

Before we give the proof of Theorem \ref{thm:small-facets}, let us 
 describe another family of polytopes with small face diameter, which  therefore satisfy the conditions of \cref{thm:small-facets}.
\begin{example}
    For $\lambda>0$  we say that a set $X \subset S^{m-1}$ is a $\lambda$-packing of $S^{m-1}$ if no two points of $X$ are at spherical distance less than or equal to $\lambda$.  We define the polytope $P_X$ as the intersection of all the half-spaces supporting $S^{m-1}$ on points of $X$ that contain $S^{m-1}$.  In other words, 
\[
P_X =\{y: \langle y, x \rangle \le 1 \quad \mbox{for all }x \in X\}.
\]

If we project each facet of $P$ onto $S^{m-1}$, we obtain the spherical Voronoi diagram induced by $X$.  We claim that if $X$ is an inclusion-maximal $\lambda$-packing the diameter of each cell in its spherical Voronoi diagram is at most $2\lambda$.  If this was not the case, first notice that each point $y \in S^{m-1}$ must be at spherical distance at most $\lambda$ from a point of $X$, or it would contradict the inclusion-maximality of $X$ (as we would be able to add $y$ to $X$).  If there was a cell $C$ of diameter greater than $2\lambda$, take two points $y_1, y_2$ at distance greater than $2\lambda$ in $C$.  If $x \in X$ is the point corresponding to $C$, then the triple $x,y_1,y_2$ would break the triangle inequality. 

Thus choosing  $\lambda < \pi/4r$, the face diameter of  $P_X$ is smaller than $\pi/r$.

%Moreover, by the same argument as above, we can remove the condition of $r$ being prime when $r$ is large enough with respect to $d$. That is, if $d,r$ are a  positive integers such that $\frac{r-2}{d} > r^{7/11}$,  $m\ge (r-1)(d+1)+1$ and $\lambda < \pi/4r$, and   $X$ is an inclusion-maximal $\lambda$-packing of $S^{m-1}$, then $P_X$ satisfies \cref{gentverberg} with an improved dimension.
\end{example}

Theorem \ref{thm:small-facets}  is a direct consequence of the following theorem, which in the case $p=2$ is a weaker version of the Borsuk--Ulam theorem.  The reason for the discrepancy in dimensions with respect to the Borsuk--Ulam theorem is the use of the Stiefel manifold, which is needed for all $p>2$.
In \cite{Yang:1957} a similar result was proven for $p=3$, with an improved condition on the dimension of the sphere: $d+1$ instead of $2d+1$. This may imply that this condition can be improved in general.

\begin{theorem}\label{thm:BU}
    Let $p$ be a prime, $d$ a positive integer, and $m \ge d(p-1)+1$.  Then for any continuous maps $f:S^{m} \to \mathbb{R}^d$ there exist $p$ points $x_1, \dots, x_p\in S^m$ on a great circle of $S^m$ such that their pairwise distance on the sphere is greater than or equal to $2\pi/p$ and $f(x_1) = \dots = f(x_p)$.
\end{theorem}

%Our proof gives a bit more, we get a $p$-tuple $x_1, \dots, x_p$ on a great circle of $S^m$.

\begin{proof}
    Consider the Stiefel manifold $V_{m+1,2}$ of orthonormal $2$-frames in $\rr^{m+1}$:
    \[
    V_{m+1,2} = \{(x,y) : x,y \in S^m,\ \langle x, y \rangle = 0\}
    \]
    Note that $V_{m+1,2}$ has a free action of $\zz_p$ defined as follows.  For a generator $g$ of $\zz_p$ and $(x,y) \in V_{m+1,2}$, we define
    \[
    g(x,y) = \left(\cos\left(\frac{2\pi}{p}\right)x + \sin\left(\frac{2\pi}{p}\right) y , -\sin\left(\frac{2\pi}{p}\right)x + \cos\left(\frac{2\pi}{p}\right) y \right).
    \]
    In other words, we rotate the pair $(x,y)$ an angle of $2\pi/p$ in their two-dimensional span.

    For a pair $(x,y)$ let $x_1, \dots, x_p$ be the first vectors of each of the pairs $(x,y)$, $g(x,y)$, $g^2(x,y)$,$\dots$, $g^{p-1}(x,y)$.

    Let $W^{p-1} = \{(y_1,\dots, y_p) \in \rr^p: y_1 + \dots + y_p=0\}\subset \rr^p$.  This space has an action of $\zz_p$ by rotating the entries that is free everywhere except at zero (which is a fixed point of the group).  
    We denote by $\pi$ the orthogonal projection of $\rr^{pd}\cong (\rr^p)^d$ onto $(W^{p-1})^d$.

    Consider the function
    \begin{align*}
        h:V_{m+1,2} & \to \rr^{pd} \\
        (x,y) & \mapsto (f(x_1),f(x_2),\dots,f(x_p))
    \end{align*}

    The function $h$ is continuous and $\zz_p$-equivariant.  Note that $f(x_1) = \dots = f(x_p)$ if an only if $\pi \circ h (x,y)$ is zero.  Moreover, by construction the pairwise spherical distance of $x_1, x_2, \dots, x_p$ is at least $2\pi/p$, and $x_1, x_2, \dots, x_p$ lie on a great circle of $S^m$.

    If $\pi \circ h$ has no zeros, we can project the image from $\pi \circ h$ to the unit sphere $S((W^{p-1})^d)$ of this space.  This would give us a $\zz_p$-equivariant continuous map $F: V_{m+1,2} \to S((W^{p-1})^d)$.  The homotopical connectedness of $V_{m+1,2}$ is $m-2$, and the dimension of $S(W^{p-1})^d$ is $d(p-1)-1$.  Since $m-2 \ge d(p-1)-1$, this  contradicts Dold's theorem. 
\end{proof}

\section{ $(d+1)$-neighborly polytopes and cyclic polytopes}\label{sec:neighborly}
In this section we show that for cyclic polytopes with at least $(r-1)(d+1)+1$ vertices, the dimension $m$ in Theorem \ref{gentverberg} can be improved to $2(d+1)$ when the map $f$ is linear.  

\begin{proposition}
    Let  $d,r$ be positive integers with $r\ge 3$, and let $P$ be a $(d+1)$-neighborly polytope with at least $(r-1)(d+1)+1$ vertices.  Then for any linear function $f:P \to \R^d$ there exist $r$ pairwise disjoint faces whose images under $f$ intersect. 
\end{proposition}
\begin{proof}
    By Tverberg's theorem, the images of the vertices in $\R^d$ admit a partition into $r$ disjoint sets, $S_1,\dots,S_r$, whose convex hulls intersect at some point $z\in \R^d$. By Carath\'eodory's theorem, for each $i\in [r]$ there exists a subset $F_i \subset S_i$ of size at most $d+1$ containing $z$. Since $P$ is $(d+1)$-neighborly, $f^{-1}(F_i)$ is a face of $P$ containing a point $x_i$ so that $f(x_i)=z$, as needed. 
\end{proof}

\begin{corollary}
      Let  $d,r$ be positive integers with $r\ge 3$, and let $P$ be a cyclic polytope of dimension at least $2(d+1)$  with at least $(r-1)(d+1)+1$ vertices.  Then for any linear function $f:P \to \R^d$ there exist $r$ pairwise disjoint faces whose images under $f$ intersect. 
\end{corollary}
\begin{proof}
      If $P$ be a cyclic polytope of dimension $\dim(P) \ge 2(d+1)$ then $P$ is $(d+1)$-neighborly (see \cite{Gale:1958}). 
\end{proof}

\section{The case $d=1$}\label{sec:dimone}
When $d=1$ Theorem  \ref{gentverberg} is true with continuous function $f$ for every (not necessarily prime power) $r$, because the image of a face $F$ in $\R$ under $f$ contains the convex hull of the image of the vertices of $F$. 
Here we show that when the 1-skeleton of $P$ is triangle-free, the dimension in Theorem \ref{gentverberg} can be improved to $r$, which is clearly best possible.

     \begin{theorem}
     If $P$ is a polytope of dimension at least $r$ such that its 1-skeleton is triangle-free. Then for any continuous function $f: P \to \R$ there exist $r$ pairwise disjoint faces of $P$ whose images under $f$ intersect. In particular, this holds for the $r$-hypercube.
\end{theorem}
\begin{proof}
   Let $P$ be a $m$-polytope with $m\ge r$ such that its 1-skeleton $G$ is triangle-free. Let $G$ the 1-skeleton of $P$. By Balinski's theorem \cite{Balinski1961} $G$ is $m$-connected, and therefore $\delta(G) \ge m \ge r$, where $\delta(G)$ is the minimum degree of $G$.  
 Let $f:P \to \R$ be a continuous function and suppose that $x_1,x_2, \dots $ are the images of the vertices of $P$ ordered from left to right on $\R$, where the order of equal points is chosen arbitrarily.   Let $v_1,v_2,\dots$ be the vertices of $P$ so that $f(v_i)=x_i$ for all $i$.  

The degree of $v_1$ in  $G$ is at least $r$ and  thus there exists  $j_1 \notin \{1,2,\dots, r\}$ such that $v_1v_{j_1}$ is an edge of $P$.  Let $e_1=v_1v_{j_1}$.

Consider $v_2$. If there exists 
$j_2 \notin \{1,2,\dots, r,j_1\}$ such that $v_2v_{j_2}$ is an edge of $P$, let $e_2= v_2v_{j_2}$.
Otherwise, $\deg_G(v_2)=r$ and the neighborhood of $v_2$ in $G$ is $N_G(v_2)= \{v_1,v_3,\dots, v_r,v_{j_1}\}.$
But this implies that $v_1,v_2,v_{j_1}$ is a triangle in $G$, contradicting the fact that $G$ is triangle-free.

Now let $3\le i\le r-1$. If there exists 
$j_i \notin \{1,2,\dots, r,j_1,j_2 \dots, j_{i-1}\}$ such that $v_{i}v_{j_{i}}$ is an edge of $P$, let $e_{i}= v_{i}v_{j_{i}}$.
Otherwise, $$N_G(v_i) \subseteq \{v_1,\dots,v_{i-1},v_{i+1},\dots, v_r,j_1, \dots, j_{i-1}\}. $$
Since $\deg_G(v_i) \ge r$, by pigeonhole principle there exists $t\in [i-1]$ such that  $v_t,v_{j_t} \in N_G(v_i)$. But this implies that  $v_i,v_t,v_{j_t}$ is a triangle in $G$, a contradiction.
Therefore,  for each $i\in [r-1]$ we can find $j_i \notin \{1,2,\dots, r,j_1,j_2 \dots, j_{i-1}\}$ such that $e_i=v_iv_{j_i}$ is an edge of $P$. 

Now, by construction we have $x_r \in \bigcap_{i=1}^{r-1} f(e_i)$.
We conclude that $e_1,\dots, e_{r-1}, v_r$ are $r$ pairwise disjoint faces of $P$ whose images under $f$ intersect.  
\end{proof}

\section{Acknowledgment} 
The authors are grateful to Pavle Blagojevi\'c, Daniel McGinnis, Eran Nevo, and Martin Tancer for useful discussions.


\end{document}